\documentclass{birkjour}
 \newtheorem{thm}{Theorem}[section]
 \newtheorem{cor}[thm]{Corollary}
 \newtheorem{lem}[thm]{Lemma}
 \newtheorem{prop}[thm]{Proposition}
 \theoremstyle{definition}
 
 \theoremstyle{remark}
 \newtheorem{rem}[thm]{Remark}
\newtheorem{ex}[thm]{Example}
 \numberwithin{equation}{section}

\usepackage{cite}

\usepackage{amsmath,color,graphicx}
\usepackage{amssymb}
\usepackage{euscript}
\newcommand{\script}[1]{\EuScript{#1}}

\theoremstyle{plain}

\newtheorem{cond}[thm]{Condition}

\DeclareMathOperator*{\argmin}{arg\,min}
\DeclareMathOperator*{\supp}{supp}
\DeclareMathOperator*{\zeros}{null}

\begin{document}

%
%
%
%
%
%
%
%
%

\title[Multidimensional Moment Problem with Complexity Constraint]
 {The Multidimensional Moment Problem with Complexity Constraint}

\author[Karlsson]{Johan Karlsson}
\address{%
Department of Mathematics \\
Royal Institute of Technology \\
SE-100 44 Stockholm \\
Sweden}

\email{johan.karlsson@math.kth.se}

\author[Lindquist]{Anders Lindquist}
\address{%
Departments of Automation  and Mathematics\\
Shanghai Jiao Tong University \\
200240 Shanghai \\
China\\
and\\
Department of Mathematics \\
Royal Institute of Technology \\
SE-100 44 Stockholm \\
Sweden}

\email{alq@kth.se}

\author[Ringh]{Axel Ringh}
\address{%
Department of Mathematics \\
Royal Institute of Technology \\
SE-100 44 Stockholm \\
Sweden}

\email{aringh@kth.se}

\thanks{This work was supported by the Swedish Foundation of Strategic Research (SSF), the Swedish Research Foundation (VR) and the Center for Industrial and Applied Mathematics (CIAM)}






\begin{abstract}
A long series of previous papers have been devoted to the (one-dimensional) moment problem with nonnegative rational measure. The rationality assumption is a complexity constraint motivated by applications where a parameterization of the solution set in terms of a bounded finite number of parameters is required. 
In this paper we provide a complete solution of the multidimensional moment problem with a complexity constraint also allowing for solutions that require a singular measure added to the rational, absolutely continuous one. Such solutions occur on the boundary of a certain convex cone of solutions. In this paper we provide complete parameterizations of all such solutions. We also provide errata for a previous paper in this journal coauthored  by one of the authors of the present paper.
\end{abstract}

\maketitle


\section{Introduction}

The multidimensional moment problem considered in this paper amounts to finding a nonnegative measure $d\mu$ on a compact subset $K$ of $\mathbb{R}^d$ solving the equations
\begin{equation}
\label{momenteqn}
c_k=\int_K\alpha_k d\mu, \quad k=1,2,\dots,n ,
\end{equation}
where $c_1,c_2,\dots,c_n$ are given numbers and  $\alpha_1,\alpha_2,\dots\alpha_n$ are given linearly independent basis functions defined on $K$.  More precisely we are interested in measures of the type
\begin{equation}
\label{complexity-constraint}
d\mu(x)=\frac{P(x)}{Q(x)}dx +d\hat{\mu}(x),
\end{equation}
where $P$ and $Q$ are nonnegative functions on $K$ formed by linear combinations of the basis functions and $d\hat{\mu}$ is a singular measure. 

Such constraints are nonclassical and motivated by applications. An important special case is to find an absolutely continuous measure 
\begin{equation}
\label{PositiveMeasure}
d\mu(x)=\frac{P(x)}{Q(x)}dx 
\end{equation}
satisfying \eqref{momenteqn}, where $P$ and $Q$ are as in \eqref{complexity-constraint}. 
Clearly \eqref{PositiveMeasure} is a complexity constraint depending on a finite number of parameters.

In the one-dimensional case, (generalized) moment problems with the complexity constraint \eqref{PositiveMeasure} has been considered in a long series of papers \cite{georgiou1999theinterpolation,BLkimura,georgiou2005solution,georgiou2006relative,byrnes2006thegeneralized,byrnes2008important,BLkrein}.  This study started with the {\em rational covariance extension problem}, which is a trigonometric moment problem with rational positive measure formulated by Kalman \cite{kalman1981realization}, i.e., $\alpha_1,\alpha_2,\dots\alpha_n$ are the trigonometric monomials, $e^{ikx}$, $k=0,1,\dots, n-1$.  In this case, given moments $c_1,c_2,\dots,c_n$ that admit solutions of \eqref{momenteqn}, it was shown in \cite{georgiou1983partial,georgiou1987realization} that there exists a  solution \eqref{PositiveMeasure} for each choice of  positive $P$, and in \cite{byrnes1995acomplete} it was established that this parameterization of the solution set is complete and smooth, i.e., the map from $P$ to $Q$ is a diffeomorphism. A constructive proof based on a certain family of convex optimization problems was given in \cite{BGuL,byrnes2001fromfinite,byrnes2002identifyability}.  These results were modified in steps to the case that  $\alpha_1,\alpha_2,\dots\alpha_n$ are Herglotz kernels in \cite{georgiou87NP,georgiou1999theinterpolation,byrnes2001ageneralized,pavon2006onthegeorgiou}, leading to Nevanlinna-Pick interpolation with positive rational measure and more general moment problems with complexity constraints \cite{GL1,BLkimura,byrnes2006thegeneralized,byrnes2008important,BLkrein,FerrantePavonRamponi2008,PavonFerrante2013}. 

In this paper we begin by generalizing certain results in \cite{byrnes2006thegeneralized} concerning moment problems over the general class of measures \eqref{PositiveMeasure} to the multidimensional case ($d>1$), but we shall also take a fresh look at the case $d=1$. However, when allowing  $P$ and $Q$ to have zeros in $K$, the class of measures have to be extended to \eqref{complexity-constraint}. Unfortunately, a key result for this case in \cite{byrnes2006thegeneralized} is incorrect, and we take the opportunity to provide a correction in this paper. 

The multidimensional moment problem is important in many applications, such as imaging, radar, sonar, and medical diagnostics \cite{milanfar1996moment,jakowatz1996spotlight,ekstrom1984digital}. A series of papers by Lang and McClellan \cite{lang1982multidimensional,lang1983spectral,mcclellan1982multi-dimensional} are of special interest to us, since in a certain sense they provide an interesting overlap with the theory described above. For this reason we shall have reasons to return to some of their results in the rest of this paper. 

The outline of the paper is as follows. In Section~\ref{sec:cones} we define the general multidimensional moment problem and introduce a set of dual cones that will be fundamental in the subsequent results. We present a generalization to the multidimensional case of a theorem by Krein and Nudelman \cite[p. 58]{krein1977themarkov} on the existence of solutions to the moment problem. In Section~\ref{sec:rational measure} we generalize some basic results in \cite{byrnes2006thegeneralized} for moment problems with rational measure to the multivariable case, and in Section~\ref{sec:optimization} we introduce the basic optimization problem, generalized to the multivariable case, and prove the basic parameterization result in the rational case. In Section~\ref{sec:boundary} we extend the parameterization to the general case when, e.g., $P$ and $Q$ are allowed to have zeros in $K$, resulting in solutions of the form \eqref{complexity-constraint}.
In Section~\ref{boundarymomentsec} we consider the case when the moments are placed on the boundary of the feasible set.  Section~\ref{appendix} is an appendix to which we have deferred some supporting results and proofs for better readability. Finally, Section~\ref{errata} contains errata for \cite{byrnes2006thegeneralized}. 


\section{The general multidimensional moment problem}\label{sec:cones}

Let $\{\alpha_1,\alpha_2,\dots,\alpha_n\}$ be a set of real-valued, continuous functions defined on the compact set $K\subset\mathbb{R}^d$. We assume that $K$ has an interior of dimension $d$, the closure of which is precisely $K$.  The functions $\alpha_1,\alpha_2,\dots,\alpha_n$ are assumed to be linearly independent, an assumption that will be retained throughout the paper. This condition holds in most interesting moment problems, as, for example, the trigonometric moment problem, Nevanlinna-Pick interpolation and the power moment problem. 


Given real numbers $c_1,c_2,\dots,c_n$, we shall consider the (truncated) moment problem to determine a bounded nonnegative measure $d\mu$ such that \eqref{momenteqn} holds. 
Whenever convenient, we shall write \eqref{momenteqn} in the vector form
\begin{equation}
\label{vectormoment}
c=\int_K\alpha d\mu,
\end{equation}
where $c=(c_1,c_2,\dots,c_n)'$ and  $\alpha=(\alpha_1,\alpha_2,\dots,\alpha_n)'$ are column vectors and prime $(')$ denotes transpose. 

The assumption that the basis functions are real can be done without loss of generality.  In fact, in the case of complex basis functions (as, for example, in the trigonometric moment problem) and complex moments, we merely exchange a complex moment equation with two real moment conditions \cite{byrnes2006thegeneralized}.

Next we define the open convex cone $\mathfrak{P}_+\subset\mathbb{R}^n$ of sequences  $p=(p_1,p_2,\dots,p_n)$ such that the corresponding generalized polynomial
\begin{equation}
\label{Pfrak}
P(x)=\sum_{k=1}^n p_k\alpha_k(x)
\end{equation}
is positive for all $x=(x_1,\dots,x_d)\in K$. Moreover, we denote by  $\bar{\mathfrak{P}}_+$ its closure and by $\partial\mathfrak{P}_+$ its boundary $\bar{\mathfrak{P}}_+\setminus\mathfrak{P}_+$. It is easy to see that $P\equiv 0$ if and only if $p=0$, since $\alpha_1,\alpha_2,\dots,\alpha_n$ are  linearly independent.


Throughout the paper we assume that the zero locus of $P$ has measure zero for each $p\in\bar{\mathfrak{P}}_+\setminus \{0\}$.   
Again many important moment problems, even in the multidimensional case, have this property.   

Moreover, we define the dual cone 
\begin{equation}
\label{Cplus}
\mathfrak{C}_+ =\{ c\in\mathbb{R}^n\mid \text{$\langle c,p\rangle >0$ for all $p\in\bar{\mathfrak{P}}_+\setminus \{0\}$}\},
\end{equation}
where $\langle c,p\rangle$ is the inner product
\begin{displaymath}
\langle c,p\rangle=\sum_{k=1}^n c_kp_k.
\end{displaymath}

\begin{prop}\label{prop:C+notempty}
The dual cone $\mathfrak{C}_+$ is nonempty if\/ $\bar{\mathfrak{P}}_+ \neq \{0\}$.
\end{prop}

\begin{proof}
Take $c = \int_K \alpha dx$. Then, for any $p \in \bar{\mathfrak{P}}_+ \setminus \{0\}$, we have
\begin{equation}\label{<c,p>geq0}
\langle c, p \rangle =  \sum_{k=1}^n \int_K \alpha_k p_k dx = \int_K P dx \geq 0 .
\end{equation}
However, due to the continuity and linear independence, $P \not \equiv 0$, and hence there is always a small neighborhood in which $P > 0$. Consequently, the inequality in \eqref{<c,p>geq0} is strict for all $p \in \bar{\mathfrak{P}}_+ \setminus \{0\}$, and hence $c\in\mathfrak{C}_+$.
\end{proof}

The dual cone $\mathfrak{C}_+$ is also an open cone, and we denote by $\bar{\mathfrak{C}}_+$ its closure and by $\partial\mathfrak{C}_+$ its boundary. 

\begin{prop}\label{mu2cprop}
Any $c\in\mathbb{R}^n$ satisfying \eqref{vectormoment} for some nonnegative measure $d\mu$ belongs to $\bar{\mathfrak{C}}_+$. 
\end{prop}

\begin{proof}
If $c$ satisfies \eqref{vectormoment}, then
\begin{displaymath}
\langle c, p\rangle =\sum_{k=1}^n p_k\int_K\alpha_k d\mu = \int_K P d\mu \geq 0
\end{displaymath}
for all $p\in\bar{\mathfrak{P}}_+$. Hence $c\in\bar{\mathfrak{C}}_+$, as claimed.  
\end{proof}

The converse also holds. The proof of the following generalization to the multidimensional case of a result in Krein and Nudelman \cite[p. 58]{krein1977themarkov} is based on Theorem~\ref{thm:mainMoment} and  will be deferred to Section~\ref{boundarymomentsec}.

\begin{thm}\label{c2muthm}
Suppose that  $\mathfrak{P}_+ \neq \emptyset$. Then for any $c \in \bar{\mathfrak{C}}_+$ there exists a bounded nonnegative measure $d\mu$  such that \eqref{vectormoment} holds. 
\end{thm}

This theorem ensures that the space 
\begin{equation}
\label{fracMc}
\mathfrak{M}_c=\left\{ d\mu\geq 0\mid \int_K\alpha\, d\mu =c\right\}
\end{equation}
of bounded measures  is nonempty for all $c \in \bar{\mathfrak{C}}_+$. However, in general, there are infinitely many solutions, and the extreme points of $\mathfrak{M}_c$ are of particular interest. 

\begin{prop}\label{dmusupportn}
Suppose that  $\mathfrak{P}_+ \neq \emptyset$.  Then, for all $c \in \bar{\mathfrak{C}}_+$ there is a $d\mu\in\mathfrak{M}_c$ that is a discrete measure with support in at most $n$ points in $K$. 
\end{prop}

\begin{proof}
Following the proof of Proposition 6 in \cite{KarlssonGeorgiou2005}, we use the Krein-Millman Theorem \cite{KreinMillman1940,Girsanov1972} to show the existence of an extreme point, and then we show that any extreme point has the claimed properties.  The space ${\mathfrak M}_c$ is the intersection of a positive cone and a closed subspace, and hence it is convex and closed. Let $p\in {\mathfrak P}_+$. Then $P(x)\geq \varepsilon$ for some $\varepsilon >0$, and hence  $\langle c,p \rangle=\int_K Pd\mu\geq \varepsilon \mu(K)$. Therefore the norm (total variation) of the elements of ${\mathfrak M}_c$ is bounded by $\langle c,p \rangle/\varepsilon$, which implies that ${\mathfrak M}_c$ is compact in the weak* topology \cite[p. 19]{Girsanov1972}. Then, since ${\mathfrak M}_c$ is a compact convex set in a locally convex topological linear space, it is the closure of the convex hull of its extreme points \cite{KreinMillman1940,Girsanov1972}. Since the set ${\mathfrak M}_c$ is nonempty (Theorem~\ref{c2muthm}), it has at least one extreme point. 

We want to prove that the extreme points of ${\mathfrak M}_c$ have support in at most $n$ points. To this end, suppose the contrary. Then there is an extreme point $d\mu=\sum_{k=1}^{n+1}\beta_k d\mu_k$ for which the measures $d\mu_k\ge0$ have distinct support and  $\beta_k>0$, $k=1,\ldots, n+1$, and therefore 
\begin{displaymath}
c=\int_K \alpha d\mu=\int_K \alpha [d\mu_1\, \cdots \, d\mu_{n+1}]\begin{bmatrix} \beta_1\\\vdots\\ \beta_{n+1}\end{bmatrix}.
\end{displaymath} 
Since $\int_K \alpha [d\mu_1\, \cdots \, d\mu_{n+1}]\in {\mathbb R}^{n\times (n+1)}$ has linearly dependent columns, there is a nontrivial affine set of solutions $\{\beta_k\}$, which contradicts the assumption that $d\mu$ is an extreme point of ${\mathfrak M}_c$. Hence the extreme points of ${\mathfrak M}_c$ has support in at most $n$ points. 
\end{proof}

\begin{cor}
The extreme points of ${\mathfrak M}_c$ have support in at most $n$ points.
\end{cor}

In several of the most important one-dimensional moment problems there are matrix tests available to check that $c\in\mathfrak{C}_+$. For example, in the trigonometric moment problem, we need to check that the corresponding Toeplitz matrix is positive definite, and, in the Nevanlinna-Pick interpolation, that the Pick matrix is positive definite. There is also a matrix test for the power moment problem in terms of a Hankel matrix \cite{krein1977themarkov}. 

In the multidimensional case, checking that $c\in\mathfrak{C}_+$ is more complicated, but the following result might provide some help. 

\begin{prop}\label{Cplustestprop}
Let $c_0\in\mathfrak{C}_+$ be arbitrary.%
\footnote{For example $c_0=\int_K \alpha dx\in \mathfrak{C}_+$ as in the proof of Proposition \ref{prop:C+notempty}.} Then the optimization problem
\begin{equation}
\label{Cplusoptimization}
\min_{p\in\bar{\mathfrak{P}}_+} \langle c,p\rangle\quad \text{subject to $\langle c_0,p\rangle=1$}
\end{equation}
has a solution with minimal value $V$. Moreover, {\rm (i)}  $V>0$ if and only if $c\in\mathfrak{C}_+$, {\rm (ii)}  $V=0$ if and only if $c\in\partial\mathfrak{C}_+$, and {\rm (iii)}  $V<0$ if and only if $c\not\in\bar{\mathfrak{C}}_+$.
\end{prop}

\begin{proof}
The constraint $\langle c_0,p\rangle=1$ is a hyperplane that does not pass through the vertex of the cone $\bar{\mathfrak{P}}_+$ where $p=0$. Since, in addition, $c_0\in\mathfrak{C}_+$, the hyperplane has a compact intersection with $\bar{\mathfrak{P}}_+$, and therefore the linear functional $\langle c,p\rangle$ has always a minimum there. Then the rest of the proposition follows directly from the definition \eqref{Cplus}. 
\end{proof}

As a preamble to the solution of the fundamental optimization problem of Section~\ref{sec:boundary}, it is instructive to consider the dual problem of \eqref{Cplusoptimization}. Differentiating the Lagrangian
\begin{displaymath}
L(p;d\mu,\lambda)=\langle c,p\rangle - \int_K Pd\mu -\lambda(\langle c_0,p\rangle -1),
\end{displaymath}
where the nonnegative bounded measure\footnote{The dual space of $C(K)$, denoted by $C(K)^*$, is the space of bounded signed measures \cite{Luenberger}.} $d\mu\in C(K)^*$ and $\lambda$ are Lagrange multipliers, we obtain
\begin{equation}
\label{Lstationarity}
c - \hat{c}-\lambda c_0 =0\quad \text{where $\hat{c}:=\int_K \alpha d\mu$}, 
\end{equation}
and the complementary slackness condition 
\begin{equation}
\label{compslack}
\langle \hat{c},p\rangle=\int_K Pd\mu =0, 
\end{equation}
which implies that $\hat{c}\in\partial\mathfrak{C}_+$. 
Since $P\geq 0$ on  $K$ and $d\mu$ is a nonnegative measure, \eqref{compslack} also implies that $d\mu$ is either identically zero or a singular measure with support in the set of zeros of $P$. Inserting the stationarity condition \eqref{Lstationarity} in the Lagrangian we get the dual functional \cite{Luenberger}
\begin{displaymath}
\begin{split}
\varphi(d\mu,\lambda)=\langle \hat{c},p\rangle +\lambda\langle c_0,p\rangle- \int_K Pd\mu -\lambda\langle c_0,p\rangle +\lambda =\lambda .
\end{split}
\end{displaymath}
 Consequently, the dual problem is 
\begin{equation}
\label{dual}
\max_{\lambda\in\mathbb{R}}\, \lambda \quad \text{subject to $c-\lambda c_0\in\partial\mathfrak{C}_+$}
\end{equation}
with optimal value $\hat{\lambda}$.
Since these are convex optimization problems and $\mathfrak{P}_+\ne\emptyset$, there is no duality gap \cite{Luenberger}, and therefore $\hat{\lambda}=V$. From this we can construct an alternative proof of Proposition~\ref{Cplustestprop}. In fact, since $\hat c\in\partial\mathfrak{C}_+$, it follows from \eqref{Lstationarity} that $c\in\partial\mathfrak{C}_+$ is equivalent to $\lambda c_0\in\partial\mathfrak{C}_+$, which in turn holds if and only if $\lambda=0$, since $c_0\in\mathfrak{C}_+$.  Increasing $\lambda$ through positive values brings $c=\hat{c}+\lambda c_0$ into $\mathfrak{C}_+$, whereas negative values of $\lambda$ brings $c=\hat{c}+\lambda c_0$ outside $\bar{\mathfrak{C}}_+$. 
This could be compared with the test procedure suggested in \cite[p. 3.2.4]{lang1981spectral}. 

Using a homotopy approach, it was shown in \cite{georgiou2005solution} that a certain differential equation has an exponentially attractive point if and only if $c\in \mathfrak{C}_+$. Such a procedure might be preferable from a computational point of view. 


\section{Solutions with rational positive measure}\label{sec:rational measure}

We begin by considering {\em rational positive measures\/} of the type
\begin{equation}
\label{rational measures}
d\mu =\frac{P}{Q}dx, \quad p, q \in\mathfrak{P}_+  
\end{equation} 
and define the moment map
\begin{equation}
\label{momentmap}
f^p(q)=\int_K \alpha\frac{P}{Q}dx . 
\end{equation}

The following condition is instrumental in ensuring the existence of an interpolating measure of the form \eqref{rational measures} for each given $p\in \mathfrak{P}_+$, which, as we will see, leads to a complete and smooth parameterization in terms of rational measures. 
\begin{cond}\label{condition2}
The cone\/ $\mathfrak{P}_+$ is nonempty and has the property
\begin{equation}
\label{boundaryinf}
\int_K \frac{1}{Q}dx =\infty \quad\text{for all $q\in\partial\mathfrak{P}_+$}.
\end{equation}
\end{cond}

The following modification of Theorem~\ref{c2muthm} is a simple corollary of Theorem~\ref{diffeomorphismthm} below, but we state it already here to motivate our interest in rational measures \eqref{rational measures}. 

\begin{prop}
Suppose that Condition~\ref{condition2} holds. Then there is a $d\mu\in{\mathfrak M}_c$ of the form \eqref{rational measures} satisfying \eqref{vectormoment} if and only if  $c\in\mathfrak{C}_+$. 
\end{prop}

This is immediate by taking any $p\in\mathfrak{P}_+$ in Corollary~\ref{rationalmomentscor}. 

\begin{lem}\label{properlem}
Suppose that Condition~\ref{condition2} holds and that $p\in\mathfrak{P}_+$. Then the map $f^p:\mathfrak{P}_+\to\mathfrak{C}_+$ is proper, i.e., the inverse image $(f^p)^{-1}(C)$ is compact for every compact $C\subset \mathfrak{C}_+$.
\end{lem}

\begin{proof}
We first prove that $(f^p)^{-1}(C)$ is bounded. To this end, first note that the set $\{(c,q)\mid c\in C, q\in\bar{\mathfrak{P}}_+, \|q\|_\infty =1\}$ is compact, and hence the bilinear form $\langle c,q\rangle$ has a minimum $\varepsilon$ there, where $\varepsilon >0$ since $C\subset \mathfrak{C}_+$.  Hence $\langle c,q\rangle\geq \varepsilon \|q\|_\infty$. However, 
\begin{displaymath}
\langle f^p(q),q\rangle =\int_K Pdx =:\kappa
\end{displaymath}
is constant. Consequently $\|q\|_\infty\leq \kappa/\varepsilon$ for any $q\in(f^p)^{-1}(C)$, proving boundedness.  If $(f^p)^{-1}(C)$ is empty or finite, it is trivially compact, so let us assume that it contains infinitely many points. Let  $(q^{k})$ be any sequence in $\mathfrak{P}_+$ such that $f^p(q^{k})\in C$ for all $k$. Since $(f^p)^{-1}(C)$ is bounded, $(q^{k})$ has a cluster point $\hat q$ in the closure $\bar{\mathfrak{P}}_+$. Compactness of  $(f^p)^{-1}(C)$ now follows from the fact that $\hat q\not\in\partial\mathfrak{P}_+$. In fact, if $\hat q\in\partial\mathfrak{P}_+$ were the case, we would have
\begin{displaymath}
  \langle f^p(\hat q),p\rangle =\int_K\frac{P^2}{Q}dx =\infty 
\end{displaymath}
by  the assumptions $p\in\mathfrak{P}_+$ and Condition~\ref{condition2}, which contradicts the fact that $\sup_{k}\langle f^p(q^{k}),p\rangle\le\sup_{c\in C}\langle c,p\rangle$ is finite. 
\end{proof}

\begin{thm}\label{diffeomorphismthm}
Suppose that Condition~\ref{condition2} holds and that $p\in\mathfrak{P}_+$.  Then the map $f^p:\mathfrak{P}_+\to\mathfrak{C}_+$ is a diffeomorphism between $\mathfrak{P}_+$ and $\mathfrak{C}_+$.  
\end{thm}

\begin{proof}
Since $p\in\mathfrak{P}_+$, the Jacobian
\begin{equation}
\label{Jacobian}
\frac{\partial f^p}{\partial q} = -\int_K \alpha\frac{P}{Q^2}\alpha' dx < 0 
\end{equation}
on all of $\mathfrak{P}_+$. In fact, for $a\in\mathbb{R}^n$, the quadratic form
\begin{displaymath}
a'\frac{\partial f^p}{\partial q} a= -\int_K (a'\alpha)^2\frac{P}{Q^2}dx =0
\end{displaymath}
if and only if $a=0$. Since $f^p$ is also proper (Lemma~\ref{properlem}), it follows from Hadamard's global inverse function theorem \cite{Hadamard} that the map $f^p$ is a diffeomorphism.
\end{proof}

\begin{cor}\label{rationalmomentscor}
Suppose that Condition~\ref{condition2} holds. Then the moment equations
\begin{equation}\label{rationalmoments}
c_k = \int_K \alpha_k  \frac{P}{Q}dx , \quad k=0,1,\dots, n
\end{equation}
have a unique solution $q\in\mathfrak{P}_+$ for each $(c,p)\in\mathfrak{C}_+\times\mathfrak{P}_+$.
\end{cor}

\begin{rem}\label{condition2rem}
As demonstrated in \cite{BLkrein}, Condition \ref{condition2} holds in the one-dimensional case if the basis functions $(\alpha_1,\alpha_2,\dots,\alpha_n)$ are Lipschitz continuous. In fact, if $q\in\partial\mathfrak{P}_+$, there is an $x_0\in K$ such that  $Q(x_0)=0$. Therefore, if $Q$ is Lipschitz continuous, there exists an $\varepsilon >0$ and a $\kappa >0$ such that $Q(x)\leq \kappa |x-x_0|$ for all $x\in [x_0-\varepsilon,x_0+\varepsilon]$, and hence 
\begin{displaymath}
\int_K\frac{dx}{Q}\geq \int_{x_0-\varepsilon}^{x_0+\varepsilon}\frac{dx}{Q} \geq \frac{1}{\kappa}\int_{x_0-\varepsilon}^{x_0+\varepsilon}\frac{dx}{|x-x_0|} =\infty.
\end{displaymath}
This is a very mild condition, since any reasonable moment problem encountered in applications would have Lipschitz continuous basis functions. Also, if $(\alpha_1,\alpha_2,\dots,\alpha_n)$ is a Chebyshev system (or T-system)%
\footnote{A set of real functions $(\alpha_1,\alpha_2,\dots,\alpha_n)$ on an interval $[a,b]$ is called a Chebyshev system if any nonzero linear combination $P(x) = \sum_{k = 1}^n p_k \alpha_k(x)$ has at most $n$ zeros \cite[p. 31]{krein1977themarkov}.}
 and contains a constant function, then after a reparameterization the basis functions will be Lipschitz continuous \cite[p. 37]{krein1977themarkov}.
\end{rem}

In the multidimensional case, the situation is a bit trickier. As was noted in \cite[p. 819]{georgiou2005solution}, Condition~\ref{condition2} always holds if $K$ is an interval in ${\mathbb R}^2$ and the basis functions $(\alpha_k)$ are twice differentiable and doubly periodic. 
However, as the following example from  \cite{lang1982multidimensional} illustrates, this does not hold for $d\ge 3$.

\begin{ex}
 Let $K:=[-\pi,\pi]^d$ and $\alpha_k(x)=\cos x_k$, $k=1,2,\dots,d$, and set
\begin{displaymath}
Q(x)=\sum_{k=1}^d (1-\cos x_k),
\end{displaymath}
which corresponds to a $q\in\partial\mathfrak{P}_+$ since $Q(0)=0$. However, this is the only zero of $Q$, and hence, in checking Condition~\ref{condition2}, we need only consider a small neighborhood $D_\varepsilon=\{ x\in K\mid \|x\|\leq \varepsilon\}$ of $x=0$.  A series expansion shows that for sufficiently small $\varepsilon >0$ we have $1-\cos x_k\geq x_k^2/4$ on all of $D_\varepsilon$, and hence 
\begin{displaymath}
\int_{D_\varepsilon}\frac{dx}{Q}\leq \int_{D_\varepsilon}\frac{dx}{\sum_{k=1}^d x_k^2/4} .
\end{displaymath}
Changing to spherical coordinates, this becomes
\begin{displaymath}
\int_{D_\varepsilon}\frac{dx}{Q}\leq \int_{\varphi_{d-1}=0}^{2\pi} \int_{\mathbf{\varphi}=0}^\pi \int_{r=0}^\varepsilon \frac{4}{r^2}r^{d-1}\sin\mathbf{\varphi}\,dr\, d\mathbf{\varphi}\,d\mathbf{\varphi}_{d-1},
\end{displaymath}
where $\mathbf{\varphi}:=(\varphi_1,\dots,\varphi_{d-2})$ and $\sin\mathbf{\varphi}:=\sin^{d-2}\varphi_1\cdots\sin\varphi_{d-2}$. This is clearly finite for $d\geq 3$, and hence Condition~\ref{condition2} does not hold in these cases. 
\end{ex}

\begin{ex}\label{nonLipschitzex}
Next consider a one-dimensional case where $\alpha$ is not Lipschitz continuous. Let $\alpha(x)=(1,x^{1/3},x^{2/3})$ and $K=[-1,1]$. Then $q=(0,0,3)\in\partial\mathfrak{P}_+$. In fact, $Q(x)=3x^{2/3}$, and hence $Q(0)=0$. However, 
\begin{equation}
\label{finiteintegral}
\int_{-1}^1\frac{dx}{Q(x)} =\int_{-1}^1\frac13 x^{-2/3}dx =2<\infty,
\end{equation}
so Condition~\ref{condition2} is not satisfied. In this case, $f^p:\mathfrak{P}_+\to\mathfrak{C}_+$ is not a diffeomorphism, so we have a counterexample to the statement of Theorem~\ref{diffeomorphismthm} with Condition~\ref{condition2} removed. In fact, $f^p$ is not even continuous in all points. Take $p:=(1,0,0)\in\mathfrak{P}_+$ and a sequence $q_k:=3(k^{-2},2k^{-1},1)\in\mathfrak{P}_+$, which tends to $q_\infty:=(0,0,3)$ as $k\to\infty$.  Then $P=1$ and $Q_k(x)=3(x^{1/3}+k^{-1})^2$, and hence $P/Q_k\to P/Q_\infty=\frac13 x^{-2/3}$ for all $x\ne 0$. However, by \eqref{finiteintegral}, $f_1^p(q_\infty)=\int_{-1}^1P/Q_\infty dx =2$, whereas 
\begin{displaymath}
f_1^p(q_k)=\int_{-1}^1\frac{P}{Q_k}dx =\frac13\int_{-1}^1(x^{1/3}+k^{-1})^{-2}dx =\int_{-1}^1\frac{x^2}{(x^2+k^{-1})^{-2}}dx  =\infty
\end{displaymath}
for all $k\geq 1$.  Consequently, $f^p$ is not  continuous in $q=(0,0,3)$. Moreover, we observe that the substitution $y=x^{1/3}$ transforms the basis to $\alpha(x)=(1,x,x^2)$, which is Lipschitz continuous. However, now $P(x)=x^2$, and hence $p\in\partial\mathfrak{P}_+$, so this is not a counterexample to Theorem~\ref{diffeomorphismthm}. We may even modify this example so that $f_1^p(q_k)<\infty$ for all finite $k$. To this end, choose $q_k:=3(k^{-2}+k^{-4},2k^{-1},1)$, which again tends to $q_\infty:=(0,0,3)$. However, it can now be shown that 
\begin{displaymath}
f_1^p(q_k)=\int_{-1}^1\frac{P}{Q_k}dx\to 2+\pi >2 =f_1^p(q_\infty)
\end{displaymath}
as $k\to\infty$, which again shows that $f^p$ is not continuous. 
\end{ex}


\section{The optimization problem}\label{sec:optimization}

Next, given the moment map \eqref{momentmap}, following \cite{BLkimura,byrnes2006thegeneralized} we construct the
$1$-form 
\begin{displaymath}
\begin{split}
\omega =\langle c-f^p(q),dq\rangle &= \sum_{k=1}^n c_kdq_k
-\int_K\sum_{k=1}^n\alpha_k\frac{P}{Q}dq_kdx \\
&=\langle c,dq\rangle -\int_K\frac{P}{Q}dQdx
\end{split}
\end{displaymath}
on $\mathfrak{P}_+$.  Taking the exterior derivative (on $\mathfrak{P}_+$) we obtain
\[
d\omega=\int_K\frac{P}{Q^2}dQ\wedge dQ dx=0,
\]
i.e., the $1$-form $\omega$ is closed.  Therefore, since $\mathfrak{P}_+$ is an open convex set and hence star-shaped in any point,  $\omega$ is exact by the Poincar{\'e} Lemma \cite[pp. 92-94]{spivak}. This means that  there exists  a smooth function $\mathbb{J}_p^c$ on $\mathfrak{P}_+$ such that 
\[
\mathbb{J}_p^c(q_1)-\mathbb{J}_p^c(q_0)=\int_{q_0}^{q_1} \omega= \int_{q_0}^{q_1} \langle c,dq\rangle -\int_{q_0}^{q_1}\int_K\frac{P}{Q}dQdx ,
\]
with the integral being independent of the path between two endpoints. Computing the path integral, one finds  that 
\begin{equation}\label{eq:objFun}
\mathbb{J}^c_p(q) = \langle c, q \rangle - \int_{K} P \log Q dx  
\end{equation}
modulo a constant of integration. Then, for each $(c,p) \in   \mathfrak{C}_+\times (\bar{\mathfrak{P}}_+ \setminus \{0\})$ we extend the functional \eqref{eq:objFun} to a map $\bar{\mathfrak{P}}_+\to\mathbb{R}\cup\{\infty\}$. 

\begin{lem}\label{striclyconvexlem}
Let $(c,p) \in  \mathfrak{C}_+\times(\bar{\mathfrak{P}}_+ \setminus \{0\})$. Then the functional $\mathbb{J}^c_p:\,\bar{\mathfrak P}_+\to\mathbb{R}$ is strictly convex. 
\end{lem}

\begin{proof}
First observe that 
\begin{displaymath}
\frac{\partial\mathbb{J}^c_p}{\partial q}= c - f^p(q),
\end{displaymath}
and consequently, in view of \eqref{Jacobian}, the Hessian $\mathbb{H}(q)$ of $\mathbb{J}^c_p$ is given by
\begin{displaymath}
\mathbb{H}(q)=-\frac{\partial f^p}{\partial q}(q) =\int_K \alpha\frac{P}{Q^2}\alpha' dx , 
\end{displaymath}
where the integrand $P/Q^2$ is nonnegative.  Since by assumption the zero locus of $P$ has measure zero,  $P/Q^2$ is zero at most on a subset of $K$ of measure zero, and consequently $\mathbb{H}(q)>0$ for the same reason as in the proof of Theorem~\ref{diffeomorphismthm}. This implies that $\mathbb{J}^c_p$ is strictly convex. 
\end{proof}

\begin{lem}\label{lowersemicontlem}
The map\/ $\mathbb{J}^c_p:\,\bar{\mathfrak P}_+\to\mathbb{R}$ is lower semi-continuous on all of $\bar{\mathfrak P}_+$.
\end{lem}

\begin{proof}
Let $(q_k)$ be a sequence in $\bar{\mathfrak{P}}_+$ converging to $q\in\bar{\mathfrak{P}}_+$ in $L_\infty$-norm.  Since the functions $(Q_k)$ and $Q$ are continuous on a compact set $K$, the convergence $Q_k\to Q$ as $k\to\infty$ is uniform. Since $Q$ is a continuous function on a compact set, $\max_x Q < \infty$.  Moreover, since the convergence is uniform, we have that $\sup_k\max_x Q_k < \infty$. Hence there is an $M$ such  that $\max_x  Q  \leq M$ and $\sup_k \max_x  Q_k \leq M$, and thus 
\begin{displaymath}
-\log \left( \frac{Q}{M} \right) \geq 0 \quad \text{and}\quad -\log \left( \frac{Q_k}{M} \right) \geq 0,\; k=1,2,\dots.
\end{displaymath}
Therefore, by Fatou's lemma,
\begin{displaymath}
-\int_K \log \left( \frac{Q}{M} \right) dx \leq \liminf_{k \to \infty} -\int_K \log \left( \frac{Q_k}{M} \right) dx
\end{displaymath}
since $Q_k \rightarrow Q$ pointwise. Consequently, $\mathbb{J}^c_p(q)\leq\liminf_{k\to\infty}\mathbb{J}^c_p(q_k)$, proving that $\mathbb{J}^c_p$ is lower semicontinuous. 
\end{proof}

\begin{lem}\label{sublevelsetlem}
The sublevel sets of $\mathbb{J}^c_p$ are compact, i.e., the inverse image $(\mathbb{J}^c_p)^{-1}(\infty,r]$ is compact  for all values of $r$.
\end{lem}

\begin{proof}
By Lemma~\ref{lem:linLog} in the appendix we have
\begin{displaymath}
r \geq \mathbb{J}^c_p(q) \geq \epsilon_c \|Q\|_\infty - \epsilon_p \log \|Q\|_\infty ,
\end{displaymath}
and by comparing linear and logarithmic growth we see that the sublevel sets are bounded from both above and below. Since they are sublevel sets of a lower semi-continuous function (Lemma~\ref{lowersemicontlem}),  they are closed. Hence they are compact.
\end{proof}

\begin{thm}\label{uniqueminthm}
Let $(c,p) \in \mathfrak{C}_+\times (\bar{\mathfrak{P}}_+ \setminus \{0\})$. Then the functional \eqref{eq:objFun} has a unique minimum in $\bar{\mathfrak{P}}_+$. Moreover, the map $g^c:\, \bar{\mathfrak{P}}_+ \setminus \{0\}\to  \bar{\mathfrak{P}}_+ $ sending $p$ to the corresponding minimizer $\hat{q}$ is continuous and injective.  
If, in addition, Condition~\ref{condition2} holds and $p\in\mathfrak{P}_+$, then the minimizer $\hat{q}$ belongs to $\mathfrak{P}_+$. 
\end{thm}

\begin{proof}
By Lemma~\ref{sublevelsetlem}  there exists a minimizer $\hat{q}\in\bar{\mathfrak P}_+$. Since $\mathbb{J}^c_p$ is strictly convex (Lemma~\ref{striclyconvexlem}), this minimizer is unique. 
Injectivity of the function $g^c$ follows from a trivially modified version of the proof of Lemma 2.4 in \cite{byrnes2006thegeneralized}. Continuity of $g^c$ is proved in the appendix (Proposition~\ref{prop:p2qhatcontinuity}). If Condition~\ref{condition2} holds and $p\in\mathfrak{P}_+$, the functional \eqref{eq:objFun} has a stationary point in the open cone $\mathfrak{P}_+$ (Theorem~\ref{diffeomorphismthm}), which must then be identical to $\hat{q}$.
\end{proof}

We note in passing that, for $(c,p) \in\mathfrak{C}_+\times\mathfrak{P}_+$, the convex optimization problem to minimize $\mathbb{J}_p^c$ over all $q\in\mathfrak{P}_+$ is the dual of the problem to maximize
\begin{equation}
\label{primalfunctional}
\mathbb{I}_p(\Phi)=\int_K P(x)\log\Phi(x)dx
\end{equation}
over all $\Phi\in{\script F}_+$ satisfying the moment condition
\begin{equation}
\label{Phimoment}
\int_K \alpha(x)\Phi(x)dx =c,
\end{equation}
where ${\script F}_+$ is the class of positive functions in $L_1(K)$. In fact, we have the following duality result.

\begin{thm}\label{thm:sec4Duality}
Suppose that $(c,p) \in \mathfrak{C}_+\times\mathfrak{P}_+$ and that Condition~\ref{condition2} holds. Then the optimization problem to maximize \eqref{primalfunctional} over all $\Phi\in{\script F}_+$ satisfying the moment condition \eqref{Phimoment} has a unique solution 
\begin{equation*}
\hat{\Phi}=\frac{P}{\hat Q},
\end{equation*}
where $\hat{q}$ is the unique minimizer of $\mathbb{J}_p^c$. Moreover,
\begin{equation*}
\mathbb{I}_p(\hat\Phi)=\mathbb{J}_p^c(\hat{q}) + \int_K P(\log P -1)dx .
\end{equation*}
\end{thm}

\begin{proof}
Given Lagrange multipliers $q=(q_1,q_2,\dots,q_n)$, form the Lagrangian 
\begin{displaymath}
\begin{split}
L(\Phi,q)&=\mathbb{I}_p(\Phi) +\sum_{k=1}^n q_k\left(c_k-\int_K \alpha_k\Phi dx\right)\\
             &=\mathbb{I}_p(\Phi) + \langle c,q\rangle - \int_K Q\Phi dx,
\end{split}
\end{displaymath}
which is finite for any fixed $q\in\mathfrak{P}_+$. 
Setting the Fr{\'e}chet differential
\begin{displaymath}
\delta L(\Phi,q;\delta\Phi) = \int_K \left(\frac{P}{\Phi} -Q\right)\delta\Phi dx =0
\end{displaymath}
for all $\delta\Phi$, we obtain the stationary point $\Phi=P/Q$, which inserted into the Lagrangian yields
\begin{displaymath}
L(P/Q,q)=\mathbb{J}_p^c(q) + \int_K P(\log P -1)dx .
\end{displaymath} 
Since $\Phi\mapsto L(\Phi,q)$ is concave for any $q\in\mathfrak{P}_+$, we have $L(\Phi,q)\leq L(P/Q,q)$ for all $\Phi\in{\script F}_+$.  However, by Theorem~\ref{uniqueminthm}, there is a unique $q\in\mathfrak{P}_+$ such that $\Phi:=P/Q$ satisfies \eqref{Phimoment}, namely $\hat q$, and hence 
\begin{displaymath}
\mathbb{I}_p(\Phi)= L(\Phi,\hat q)\leq L(P/\hat{Q},\hat q) = \mathbb{I}_p(\hat\Phi)
\end{displaymath}  
for all $\Phi\in{\script F}_+$ satisfying \eqref{Phimoment}, proving the required optimality. 
\end{proof}

The choice $P\equiv 1$ yields the {\em maximum entropy\/} solution of the moment problem. This duality has been extensively discussed in \cite{BGuL,byrnes2001ageneralized, byrnes2001fromfinite,byrnes2002identifyability,GL1,BLkimura,byrnes2008important,BLkrein} in the one-dimensional case. In the special case of trigonometric basis functions, the multidimensional case was already covered in \cite{lang1982multidimensional,lang1983spectral,mcclellan1982multi-dimensional}, but in a more general framework of weighted maximum entropy optimization which does not consider parameterization of rational solutions and related issues, something that  is important in our present context.
  
The duality result of Theorem~\ref{thm:sec4Duality} relies on Condition \ref{condition2} and will be reformulated in Section~\ref{sec:boundary}, where this condition will not be required.  In this case the optimal solution may contain a singular part, that is, be of the form \eqref{complexity-constraint} (Theorem~\ref{thm:mainDuality}).

\begin{thm}\label{gcdiffthm}
Suppose that Condition~\ref{condition2} holds and $c\in\mathfrak{C}_+$.  Then the map $g^c:\, \mathfrak{P}_+\to\mathfrak{P}_+$, restricted to $\mathfrak{P}_+$, is a diffeomorphism onto its image $\mathfrak{Q}_+$.   
\end{thm}

\begin{proof}
Since $g^c$ is continuous and injective (Theorem~\ref{uniqueminthm}), $\mathfrak{Q}_+$ is an open set of the same dimension as $\mathfrak{P}_+$. By definition, $g^c:\, \mathfrak{P}_+\to\mathfrak{Q}_+$ is also surjective. Next define the function
\begin{displaymath}
\varphi(p,q)=c-\int_K\alpha\frac{P}{Q}dx.
\end{displaymath}
Then the moment equations (stationarity condition) can be written $\varphi(p,q)=0$. Since 
\begin{displaymath}
\frac{\partial\varphi}{\partial q}= \int_K\alpha\frac{P}{Q^2}\alpha' dx
\end{displaymath}
is positive definite on all of $\mathfrak{P}_+\times\mathfrak{P}_+$, the Implicit Function Theorem implies that  $q=g^c(p)$  where $g^c$ is continuously differentiable. Moreover,
\begin{displaymath}
\frac{\partial g^c}{\partial p}(p)=-\left[\int_K\alpha\frac{P}{Q^2}\alpha' dx\right]^{-1}\frac{\partial\varphi}{\partial p}(p,g^c(p))
\end{displaymath}
is positive definite, since 
\begin{displaymath}
\frac{\partial\varphi}{\partial p}(p,q)= -\int_K\alpha\frac{1}{Q}\alpha' dx 
\end{displaymath}
is negative definite. Hence, by the Inverse Function Theorem, the inverse function $(g^c)^{-1}$ is also continuously differentiable. Consequently, $g^c:\, \mathfrak{P}_+\to\mathfrak{Q}_+$ is a diffeomorphism.
\end{proof}

Together with Corollary~\ref{rationalmomentscor}, Theorem~\ref{gcdiffthm} yields  a complete parameterization of all solutions of the rational moment equations \eqref{rationalmoments}. This generalizes to the multidimensional case the corresponding results in \cite{byrnes2006thegeneralized}, which in turn are  generalizations of  the results in \cite{byrnes1995acomplete}.

If $q\in \partial\mathfrak{P}_+$, it follows that $q\not\in \mathfrak{Q}_+:=g^c(\mathfrak{P}_+)\subset \mathfrak{P}_+$, and hence, by Theorem~\ref{gcdiffthm}, $p\not\in \mathfrak{P}_+$, and thus $p\in \partial\mathfrak{P}_+$. This yields the following corollary.

\begin{cor}
Suppose that Condition~\ref{condition2} holds, $c\in\mathfrak{C}_+$, and $P/Q$ satisfies the rational moment condition \eqref{rationalmoments}. Then $q\in\partial\mathfrak{P}_+$ implies that $p\in\partial\mathfrak{P}_+$. 
\end{cor}

Note, however, that the converse is not true. In fact, a\/ $q\in\partial\mathfrak{Q}_+$ could be contained in $\mathfrak{P}_+$ as the following simple one-dimensional example shows. Let $K=[-\pi,\pi]$,
$\alpha_1=1$ and $\alpha_2=\cos x$, and suppose that $P(x)=1-\cos x$ and $Q(x)\equiv 1$. Then $c=(1,-1/2)$. Since $c$  has a positive definite Toeplitz matrix,  $c\in\mathfrak{C}_+$. Moreover, since $P(0)=0$, $p\in\partial\mathfrak{P}_+$, and therefore, by Theorem~\ref{gcdiffthm}, $q=(1,0)\in\partial\mathfrak{Q}_+$. However, clearly $q\in\mathfrak{P}_+$. 

Moreover, again Condition~\ref{condition2} is crucial. In fact, in Example~\ref{nonLipschitzex}, where this condition does not hold, $q\in\partial\mathfrak{P}_+$ whereas $p\in\mathfrak{P}_+$ is in the interior. However, under the variable substitution $y=x^{1/3}$, which makes the basis Lipschitz continuous so that Condition~\ref{condition2} holds, $p$  moves to the boundary $\partial\mathfrak{P}_+$.


\section{Solutions on the boundary}\label{sec:boundary}

If Condition~\ref{condition2} holds and $p\in\mathfrak{P}_+$, then $\mathbb{J}^c_p$ has a unique minimum in the open cone $\mathfrak{P}_+$, which solves the  moment equations \eqref{rationalmoments}. On the other hand, if these conditions are not satisfied, the minimizer  may end up on the boundary $\partial\mathfrak{P}_+$, leading to complications  described in \cite{nurdin2006tac,nurdin2006new} for the special case of rational covariance extension. Therefore, in the present more general situation, the constraint $Q(x)\geq 0$ becomes essential for solving the optimization problem.  

\begin{thm}\label{thm:mainMoment}
Let $(c,p) \in  \mathfrak{C}_+\times(\bar{\mathfrak{P}}_+ \setminus \{0\})$. Then there exists a unique pair $(\hat{c},\hat{q}) \in \partial \mathfrak{C}_+\times (\bar{\mathfrak{P}}_+ \setminus \{0\})$ such that 
\begin{subequations}\label{generalmomenteqn}
\begin{equation}\label{generalmomenteqna}
c_k = \int_K \alpha_k  \frac{P}{\hat Q}dx + \hat{c}_k, \quad k=0,1,\dots, n.
\end{equation}
Here
\begin{equation}\label{mu2chat}
\hat{c}_k=\int_K \alpha_k d\hat{\mu} ,\quad k=0,1,\dots, n, 
\end{equation}
\end{subequations}
where $d\hat{\mu}$ is a (not necessarily unique) singular bounded nonnegative measure such that $\supp(d\hat{\mu}) \subset \zeros(\hat Q)$, i.e., the support of the discrete measure $d\hat{\mu}$ is contained in the set of zeros of $\hat Q$. The vector $\hat q$ is the unique minimizer of \eqref{eq:objFun} over $\bar{\mathfrak{P}}_+$. 
\end{thm}

\begin{proof}
Since $Q\in C(K)$, Lagrange relaxation leads to the Lagrangian
\begin{equation}
\label{Lagrangian}
L(q,\mu)=\mathbb{J}^c_p(q) -\int_K Q d\mu ,
\end{equation}
where $q\in\bar{\mathfrak{P}}_+$, and where the Lagrange multiplier $d\mu\in C(K)^*$ is a nonnegative measure.
Then setting 
\begin{displaymath}
\frac{\partial L}{\partial q_k}(q,\mu)=c_k-\int_K \alpha_k\frac{P}{Q}dx - \int_K \alpha_k d\mu, \quad k=1,2,\dots,n, 
\end{displaymath}
equal to zero, we obtain \eqref{generalmomenteqn} for the saddle point $(\hat q, \hat \mu)$ \cite{Luenberger}.  By Theorem~\ref{uniqueminthm}, the minimizer $\hat q$ is unique. Then it is seen from  \eqref{generalmomenteqna} that $\hat c$ is also unique.
Now, by complementary slackness \cite[Theorem 1, p. 217]{Luenberger}, 
\begin{equation}
\label{complslackness}
\int_K \hat Q d\hat\mu = 0.
\end{equation}
Since the optimization problem is convex with a strictly feasible point,  the KKT conditions (i) $\hat q\in\bar{\mathfrak{P}}_+$,  (ii) $d\hat\mu$ nonnegative bounded measure in $C(K)^*$, (iii) \eqref{generalmomenteqn} and (iv) \eqref{complslackness}  are necessary and sufficient for optimality. Since $\hat Q\geq 0$ and  $d\hat\mu\geq 0$ on $K$, \eqref{complslackness} can only hold if the support of $d\hat\mu$ is contained in the set of zeros of $\hat Q$. However, such zeros exist only if $\hat q\in\partial\mathfrak{P}_+$. 

It remains to show that $\hat c \in\partial \mathfrak{C}_+$. To this end, first note that, in view of the representation \eqref{mu2chat}, $\hat c\in\bar{\mathfrak{C}}_+$ (Proposition~\ref{mu2cprop}). 
Moreover, by \eqref{complslackness},
\begin{displaymath}
\langle \hat c,\hat q\rangle = \int_K\hat Q d\hat\mu =0,
\end{displaymath}
and hence $\hat c \in\partial \mathfrak{C}_+$. 
\end{proof}

We are now in a position to generalize the duality relation from Theorem \ref{thm:sec4Duality} to the more general setting where Condition~\ref{condition2} is no longer required. Thus we extend the domain of the objective function $\mathbb{I}_p$ of the primal problem to include any nonnegative measure $d\mu$ and define
\begin{equation}
\label{primalfunctionalMeasure}
\mathbb{I}_p(d\mu)=\int_K P(x)\log\Phi(x)dx ,
\end{equation}
where $d\mu = \Phi(x)dx + d\hat{\mu}$ is the unique Lebesgue decomposition of the measure  \cite{rudin1987real}. 
%
\begin{thm}\label{thm:mainDuality}
Let  $(c,p) \in  \mathfrak{C}_+\times(\bar{\mathfrak{P}}_+ \setminus \{0\})$ be given. Then the problem to  maximize \eqref{primalfunctionalMeasure} over the set of nonnegative measures of bounded variation, subject to the moment condition \eqref{momenteqn}, has a solution on the form 
\[
d\mu = \frac{P(x)}{\hat{Q}(x)}dx + d\hat{\mu},
\] 
where $\hat q\in \bar{\mathfrak{P}}_+ \setminus \{0\}$ and $d\hat{\mu}$ is a  singular bounded nonnegative measure such that $\supp(d\hat{\mu}) \subset \zeros(\hat Q)$.
\end{thm}

\begin{proof}
Following the same path as in the proof of Theorem \ref{thm:sec4Duality}, we now consider the Lagrangian 
\begin{displaymath}
\begin{split}
L(d\mu, q)&=\mathbb{I}_p(d\mu) +\sum_{k=1}^n q_k\left(c_k-\int_K \alpha_k (\Phi dx + d\hat{\mu})\right)\\
             &=\int_K P(x)\log\Phi(x)dx + \langle c,q\rangle - \int_K Q\Phi dx - \int_K Qd\hat{\mu}.
\end{split}
\end{displaymath}
Note that for $(d\mu, q)$ to be a saddle point it is necessary that $q \in \bar{\mathfrak{P}}_+$ and $\supp(d\hat{\mu}) \subset \zeros(Q)$. This means that the last term must disappear for any saddle point candidate,
and we are therefore left with a function that is identical to the Lagrangian in the proof of Theorem \ref{thm:sec4Duality}. The proof then follows along the same lines as  that of Theorem \ref{thm:sec4Duality}. Also note that the existence of $d\hat \mu$ is ensured by Theorem \ref{thm:mainMoment}.
\end{proof}

It is interesting to note that the functional \eqref{primalfunctionalMeasure} is concave but not strictly concave, since the value does not depend in the singular part $d\hat{\mu}$. Therefore it is not surprising that  the optimal singular measure $d\hat{\mu}$ is not guaranteed to be unique. Moreover the function \eqref{primalfunctionalMeasure} can in fact be seen as a Kullback-Leibler-like divergence index between the two measures $dp(x) := P(x)dx$ and $d\mu(x)$ \cite{GL1}. In fact, maximizing \eqref{primalfunctionalMeasure} is equivalent to minimizing $\int_{K} P \log \left( P/\Phi \right) dx$, 
and, since $dp$ is absolutely continuous with respect to $d\mu$, we have the Kullback-Leibler-like divergence 
\begin{displaymath}
\mathbb{S}_{KL}(dp\| d\mu ) := \int_{K} \log \left( \frac{dp}{d\mu} \right) dp = \int_{K} P \log \left( \frac{P}{\Phi} \right)dx,
\end{displaymath}
where $(dp/d\mu) = P/\Phi$ is the Radon-Nikodym derivative \cite[p. 553-554]{renyi1961measures}.

For later reference we collect the KKT conditions of Theorem~\ref{thm:mainMoment} in the following corollary. 

\begin{cor}\label{KKTcor}
Let $c \in \mathfrak{C}_+$. Then $\hat{q}$ is the optimal solution to 
\begin{equation}
 \min_{q \in \bar{\mathfrak{P}}_+} \mathbb{J}^c_p(q) \label{thm1:optimizer}
\end{equation}
if and only if
\begin{subequations}\label{thm1:2}
\begin{eqnarray}
&&\hat{q} \in \bar{\mathfrak{P}}_+, \quad \hat c \in\partial \mathfrak{C}_+\label{thm1:feas}\\
&&c_k = \int_K \alpha_k \frac{P}{\hat Q}dx + \hat{c}_k,\quad k=0,1,\dots,n  \label{thm1:KKT}\\
&&\langle \hat{c}, \hat{q} \rangle = 0 \label{thm1:slack}
\end{eqnarray}
\end{subequations}
\end{cor}

Note that $\hat c=0$ whenever $\hat q\in \mathfrak{P}_+$, since then $\hat Q(x)>0$ on all of $K$. This is the situation of Corollary~\ref{rationalmomentscor}. If $\hat q\in\partial\mathfrak{P}_+$, we have $\supp(d\hat{\mu}) \subset \zeros(\hat Q)$. Although $\hat{c}$ in \eqref{mu2chat} is unique, the measure $d\hat{\mu}$ may not be unique in general, as can be seen from the following examples.

\begin{ex}
Next we provide a one-dimensional example where $d \hat{\mu}$ is not unique. Let $K=[0,1]$, $\alpha_1=1$ and $\alpha_2(x) = (1-x)\big( \cos\big( \frac{x}{1-x} \big) +1 \big)$. First note that $\alpha_2(1):=\lim_{x\to 1} \alpha_2(x) = 0$, and thus $\alpha_2$ is continuous on all of $K$. Since $\alpha_2(x) \in [0,2]$ for all $x\in K$, we have $\bar{\mathfrak{P}}_+=\{ q\in\mathbb{R}^2\mid q_1\geq 0, q_2\geq -\frac12 q_1\}$. Now take $c=(2,\gamma)'$, where $\gamma:=\int_K\alpha_2dx$. Since $\alpha_2\geq 0$ but not identically zero, $\gamma >0$. Moreover,
\begin{displaymath}
\gamma < \int_0^1\cos\left(\frac{x}{1-x}\right)dx +1 <2.
\end{displaymath}
Hence $0<\gamma<2$. Let $q\in\bar{\mathfrak{P}}_+\setminus\{0\}$ be arbitrary. If $q_1=0$, we have $q_2>0$, and hence $\langle c,q\rangle = \gamma q_2 >0$. If $q_1>0$, we have $\langle c,q\rangle =2q_1+\gamma q_2\geq \big(2-\frac{\gamma}{2}\big)q_1>0$. Therefore $\langle c,q\rangle>0$ for all $q\in\bar{\mathfrak{P}}_+\setminus\{0\}$, and hence $c\in\mathfrak{C}_+$. Next taking $P=\hat{Q}=\alpha_2$, we have $P(1)=\hat{Q}(1)=0$, i.e.,  $p=\hat{q}=(0,1)'\in\partial\mathfrak{P}_+$. Then forming
\begin{displaymath}
\begin{split}
\hat{c}_1 &=c_1 -\int_0^1\alpha_1\frac{P}{\hat{Q}}dx = 2 -1 =1\\
\hat{c}_2 &=c_2 -\int_0^1\alpha_2\frac{P}{\hat{Q}}dx = \gamma -\gamma =0
\end{split}
\end{displaymath} 
$\langle\hat{c},q\rangle=q_1\geq 0$ for all $q\in\bar{\mathfrak{P}}_+\setminus\{0\}$, and  hence $\hat{c}\in\bar{\mathfrak{C}}_+$. Moreover, $\langle\hat{c},\hat{q}\rangle=\hat{q}_1=0$, so $\hat{c}\in\partial\mathfrak{C}_+$. Consequently, by Corollary~\ref{KKTcor}, $\hat{q}$ is the minimizer of $\mathbb{J}_p^c$, and hence the support of the measure $d\hat{\mu}$ in \eqref{mu2chat} is contained in $\zeros(\hat{Q})=\zeros(\alpha_2)$. However $\alpha_2(t)$ have infinitely many zeros,  and any measure $d\mu=dx +d\hat{\mu}$ such that $\hat{\mu}(K)=1$ and $\supp(d\hat{\mu}) \subset \zeros(\hat Q)$ is a solution.
\end{ex}

\begin{ex}\label{2Dcancelationex}
It is easy to construct a small two-dimensional example for which the measure $d\hat{\mu}$, representing  $\hat c$, is not unique. Let $K=[0, 1]^2$ and $\alpha(x)=(1, x_1,  x_2, x_1x_2)'$, and define 
\begin{displaymath}
c=\int_K\alpha\frac{P}{\hat Q}dx +\hat{c},
\end{displaymath}
where $P(x)=x_1$, $\hat Q = x_1(1+x_2)$, and $\hat{c}=\int_K\alpha \delta(x_1)dx = (1, 0, 1/2, 0)'$.  Clearly $c\in\mathfrak{C}_+$ and $\hat{q} \in \bar{\mathfrak{P}}_+$. Moreover, $\hat{c}\in\partial\mathfrak{C}_+$. In fact,
\begin{displaymath}
\langle \hat{c},\hat{q}\rangle =\int_K \hat Qd\hat{\mu} = \int_{0}^1  x_1\delta(x_1)dx_1  \int_{0}^1 (1+ x_2)dx_2 =0.
\end{displaymath}
Therefore $(\hat c,\hat q)$ satisfies the KKT  conditions \eqref{thm1:2} and is the unique minimizer of $\mathbb{J}_p^c$ (Corollary~\ref{KKTcor}). However, $\zeros (\hat Q)$ is the whole line  $x_1=0$, and any measure $d\hat{\mu}$ with  mass $1$ and support constrained to $x_1=0$ such that  $\int_K x_2 d\hat{\mu}=1/2$ is a solution. 
Hence there are infinitely many ways to select $d\hat{\mu}$.  
\end{ex}

%
%

\begin{cond}\label{condition4}
The vectors $\alpha(x_1),\alpha(x_2),\dots,\alpha(x_m)$ are linearly independent, where  
$x_1,x_2,\dots,x_m$ are the points where the optimal polynomial \eqref{generalmomenteqn} have zeros, i.e., $\hat Q(x_j)=0$.
\end{cond}

\begin{prop}\label{linearindependenceprop}
Suppose that Condition~\ref{condition4} holds.
Then the measure $d\hat{\mu}$ in Theorem \ref{thm:mainMoment} is unique. Moreover, 
\begin{equation}
\label{chatdelta}
d\hat{\mu} = \sum_{j=1}^m a_j\delta(x-x_j)dx
\end{equation}
for some $a_1,\dots,a_m\in\mathbb{R}^n$, where $m\leq n$. 
\end{prop}

\begin{proof}
Inserting \eqref{chatdelta} into \eqref{mu2chat} yields
\begin{displaymath}
\hat{c}_k=\sum_{j=1}^m \alpha_k(x_j)a_j,
\end{displaymath}
which has a unique solution $(a_1,a_2,\dots,a_m)$ if $\alpha(x_1),\alpha(x_2),\dots,\alpha(x_m)$ are linearly independent.
\end{proof}

\begin{rem} 
In the one-dimensional case ($d=1$),  $\alpha(x_1),\alpha(x_2), \ldots, \alpha(x_m)$ are linearly independent for all distinct points $x_1,x_2,\dots,x_m$ such that $m\leq n-1$ in any T-system \cite{karlin1966tchebycheff,krein1977themarkov}, for example, the trigonometric and power moment problems, and also the Herglotz basis used in Nevanlinna-Pick interpolation. 
In these cases the zero set of $Q$ will always satisfy the linear independence property of Proposition~\ref{linearindependenceprop}, resulting in a unique $d\hat{\mu}$. 
\end{rem}

\begin{thm} 
 Let $(c,p) \in  \mathfrak{C}_+\times(\bar{\mathfrak{P}}_+ \setminus \{0\})$. Suppose that Condition~\ref{condition4} holds for the minimizer $\hat{Q}$ of \eqref{thm1:optimizer}
and let $d\mu=P/\hat{Q}dx+d\hat\mu$ be the unique corresponding measure in \eqref{thm1:2}. Moreover, let  $(p_k)$ be a sequence in $\bar {\mathfrak P}_+$ such that $p_k\to p$ as $k\to \infty$, and let $d\mu_k=(P_k/Q_k)dx+d\hat \mu_k$ be measure \eqref{thm1:2} corresponding to the minimizer of\/ ${\mathbb J}_{p_k}^c$. Then  $d \mu_k\to d \mu$ in weak$^*$  as $k\to \infty$.
\end{thm}
 
\begin{proof}
We want to show that   
\begin{equation*}
\int_Kfd\mu_k \to \int_K fd\mu
\end{equation*}
for an arbitrary $f\in C(K)$. We may choose $\rho(x)=r'\alpha(x)$ so that $\rho(x)=f(x)$ for $x\in {\rm null} (\hat Q)$.  In fact, since  $\alpha(x_1),\alpha(x_2),\dots,\alpha(x_m)$ are linearly independent (Condition~\ref{condition4}), it follows that the system of linear equations  
\begin{displaymath}
r'\begin{bmatrix}\alpha(x_1)&\alpha(x_2)&\cdots&\alpha(x_m)\end{bmatrix}=\begin{bmatrix}f(x_1)&f(x_2)&\cdots&f(x_m)\end{bmatrix}
\end{displaymath}
has a  solution. Then, setting $g=f-\rho$,   ${\rm null}(\hat Q)\subset {\rm null}(g)$. Moreover, since $d \mu_k$ and $d \mu$ both satisfy the moment condition \eqref{vectormoment}, we have 
$\int_K\rho\, d \mu_k=\int_K\rho \, d \mu = \langle c, r\rangle$, and hence it is sufficient to show that 
 \begin{equation}\label{toshown}
\int_Kg\, d\mu_k \to \int_K g\, d\mu.
\end{equation}
Next, fix $\epsilon>0$, and choose $M$ so that $\sup_k \mu_k(K)\le M$ and $ \mu(K)\le M$. Let $B_\delta :=\{x_0+x_1\in K\mid x_0\in{\rm null}(\hat Q),\,\|x_1\|_2<\delta\}$, where $\delta>0$ is chosen so that $|g(x)|<\epsilon/(2M)$ on  $B_\delta$, which can be done since $g$ is  continuous and $g(x)=0$ on ${\rm null}(\hat Q)$. By Theorem~\ref{thm:mainMoment}, ${\rm supp}(d\hat \mu) \subset B_\delta$.  Also,  since $\hat Q_k\to \hat Q$ uniformly (Theorem \ref{uniqueminthm}), it follows that,  for $k$ sufficiently large,  ${\rm null}(\hat Q_k) \subset B_\delta$, and hence,  ${\rm supp}(d\hat \mu_k) \subset B_\delta$ (Theorem~\ref{thm:mainMoment}). Thus for $k$ large enough, we have
\begin{align}
&\left|\int_K g(d\mu_k-d\mu)\right|=\left|\int_{B_\delta} g(d\mu_k-d\mu)+\int_{K\setminus B_\delta} g(P_k/\hat{Q}_k-P/\hat{Q})dx\right|\nonumber \\
&\qquad \le2M\max_{x\in B_\delta} |g(x)|+\int_{K\setminus B_\delta} |g| dx \max_{x\in K\setminus B_\delta}|P_k/\hat{Q}_k-P/\hat Q|\label{eq:TermLimit2}.
\end{align}
The first term is bounded by $\epsilon$ (by the definition of $B_\delta$) and the second term tends to zero since $P_k/\hat{Q}_k\to P/\hat{Q}$ uniformly on $K\setminus B_\delta$. Since $\epsilon>0$ is arbitrary, the limit  of $\left|\int_K g(d\mu_k-d\mu)\right|$ as $k\to \infty$ of \eqref{eq:TermLimit2} is zero, and weak$^*$ convergence $d\mu_k\to d\mu$ follows.
\end{proof}

In many classical one-dimensional moment problems, such as the power moment problem and the trigonometric moment problem, there will be cancellation of common factors in $P$ and $Q$. With more general basis functions this is not necessarily the case. In multidimensional  generalizations of the classical problems such cancelation may or may not occur. Example~\ref{2Dcancelationex} shows a situation where there is cancelation, whereas no cancellation occurs in the next simple example. 

\begin{ex}
Consider a two-dimensional power moment problem on $K = [0,1]^2$ with basis functions $\alpha(x)=(1, x_1, x_1^2, x_2, x_2^2, x_1x_2)'$. The polynomials $P(x) = x_1^2 + x_2^2$ and $Q(x) = x_1 + x_1^2 + x_2 + x_2^2$ are irreducible with a common zero in $x=(0,0)'$. Moreover,
\begin{displaymath}
\lim_{x \to 0} \frac{P(x)}{Q(x)} = 0,
\end{displaymath}
showing that it is integrable. Set  $c := \int_K\alpha\frac{P}{Q}dx$. Then $\langle c,r\rangle= \int_K R\frac{P}{Q}dx>0$ for all $r\in\bar{\mathfrak{P}}_+\setminus\{0\}$, so $c\in\mathfrak{C}_+$. Consequently, we have an example where $\hat{c}=0$ and both $p$ and $q$ belong to $\partial\mathfrak{P}_+$, but there is no cancellation. 
\end{ex}


%


\section{Moments on the boundary}\label{boundarymomentsec}

If $c\in\partial\mathfrak{C}_+$, there is a $q_0\in\bar{\mathfrak{P}}_+ \setminus \{0\}$ such that $\langle c,q_0\rangle =0$. Then $\mathbb{J}^c_{p}(\lambda q_0)\to -\infty$ as $\lambda\to \infty$. Consequently, the functional \eqref{eq:objFun} has no minimum.  Then $d\mu$ in \eqref{momenteqn} cannot have a rational part. In fact, if $d\mu$ is given by \eqref{complexity-constraint}, then 
\begin{equation}
\label{<c,q0>}
\langle c,q_0\rangle = \int_K Q_0\frac{P}{Q}dx +\int_K Q_0d\hat{\mu}=0.
\end{equation}
Since the first term is positive and the second in nonnegative, there could be no rational part in $d\mu$. Moreover, $d\hat{\mu}$ must have support in $\zeros(Q_0)$. More precisely, we have the following representation (c.f., Appendix A \cite{mcclellan1982multi-dimensional}).

\begin{prop}
Suppose that\/ ${\mathfrak P}_+\ne\emptyset$. Then, for any $c \in \partial \mathfrak{C}_+$, there exists a $d \mu\in\mathfrak{M}_c$ with support in at most $n-1$ points in $K$.
\end{prop}

\begin{proof}
Since $c\in\bar{\mathfrak C}_+$, it follows from Proposition~\ref{dmusupportn} that there is a $d \mu\in\mathfrak{M}_c$ with support in at most  $n$ points $x_1,x_2,\dots,x_n$. Then 
$d\mu=\sum_{\ell=1}^{n}\beta_\ell d\mu_\ell$ for some nonnegative coefficients $\beta_\ell$, where $d\mu_\ell=\delta(x-x_\ell)$ is the Dirac measure, for $\ell=1,2,\dots,n$. If $\beta_\ell=0$ for some $\ell$, then $d\mu$ has support in at most $n-1$ points, so only the case that $\beta_\ell>0$  for $\ell=1,2,\dots,n$, remains.  Now, since $c\in\partial{\mathfrak C}_+$, there is a $p\in \bar{\mathfrak P}_+\setminus\{0\}$ such that $\langle c,p \rangle=0$, i.e.,
\begin{displaymath}
0=\langle c,p \rangle = p'\int_K \alpha \sum_{\ell=1}^n \beta_\ell d\mu_\ell=p'\sum_{\ell=1}^n \alpha(x_\ell) \beta_\ell
\end{displaymath}
and hence $\{\alpha(x_1),\ldots, \alpha(x_n)\}$ are linearly dependent. Therefore one of the measures $d\mu_\ell$ can be eliminated in the representation 
\begin{displaymath}
c=\int_K \alpha \sum_{\ell=1}^n \beta_\ell d\mu_\ell=\sum_{\ell=1}^n\alpha(x_\ell) \beta_\ell,
\end{displaymath}
proving that only a measure $d\mu$ with support in $n-1$ is needed. 
\end{proof}

Next we provide the deferred proof of Theorem~\ref{c2muthm}, which is based on Theorem~\ref{thm:mainMoment} only.

\begin{proof}[Proof of Theorem~\ref{c2muthm}]
From Theorem~\ref{thm:mainMoment} we know that a solution of \eqref{vectormoment} exists for all $c \in \mathfrak{C}_+$. Hence it just remains to show that there is a solution for all $c \in \partial \mathfrak{C}_+$. 
For any $c \in \partial \mathfrak{C}_+$, let $c_\varepsilon = c + \varepsilon c_0$,  where $c_0 \in \mathfrak{C}_+$. The existence of such a $c_0$ is insured by Proposition~\ref{prop:C+notempty} since $\mathfrak{P}_+ \neq \emptyset$.  Then $c_\varepsilon\in\mathfrak{C}_+$ for all $\varepsilon > 0$. Therefore, by Theorem~\ref{thm:mainMoment} and the fact that $\mathfrak{P}_+ \neq \emptyset$, we know that there exist a bounded measure $d \mu_\varepsilon\in C(K)^*$ of the form  $d \mu_\varepsilon = (P/Q_\varepsilon) dx + d\hat{\mu}_\varepsilon$ such that
\begin{displaymath}
c_\varepsilon = \int_K \alpha d\mu_\varepsilon
\end{displaymath}
for some $p \in \bar{\mathfrak{P}}_+ \setminus \{0\}$. Now, since $\mathfrak{P}_+ \neq \emptyset$, there is a $p_0\in\mathfrak{P}_+$, and there is a $\delta >0$ such that $P_0(x)\geq\delta$ for all $x\in K$.   Then
\begin{displaymath}
\langle c_\varepsilon, p_0 \rangle = \int_K P_0 d\mu_\varepsilon \geq \mu_\varepsilon(K)\delta,
\end{displaymath}
However, $\langle c_\varepsilon, p_0 \rangle=\langle c, p_0 \rangle +\varepsilon\langle c_0, p_0 \rangle \leq M$ for some $M$ and sufficiently small $\varepsilon$. Hence $\mu_\varepsilon(K)\leq M/\delta$ for all such $\varepsilon$. Therefore there is a subsequence $\{d\mu_{\varepsilon_j}\}$ that converges to some $d\mu$ in $\text{weak}^*$ \cite[p. 128]{Luenberger}, \cite[p. 246]{rudin1987real}, and consequently 
\begin{displaymath}
c = \lim_{\varepsilon \rightarrow 0} c_\varepsilon = \lim_{j\to\infty}  \int_K \alpha d\mu_{\varepsilon_j}= \int_K \alpha d\mu
\end{displaymath}
which proves Theorem~\ref{c2muthm}.
\end{proof}

\section{Appendix}\label{appendix}

The following lemma, showing that $\mathbb{J}^c_p$ is bounded from below,  is a direct generalization of Proposition 2.1 in \cite{byrnes2006thegeneralized}, and the proof follows along the same lines as in  \cite{byrnes2006thegeneralized}. 

\begin{lem}\label{lem:linLog}
Let $(c,p)\in\mathfrak{C}_+\times \bar{\mathfrak{P}}_+\setminus\{0\}$. Then there are constants $\epsilon_c, \, \epsilon_p > 0$ such that 
\begin{equation}\label{Jestimate}
\mathbb{J}^c_p(q) \geq \epsilon_c \|Q\|_\infty - \epsilon_p \log\|Q\|_\infty 
\end{equation}
for all $q\in \bar{\mathfrak{P}}_+\setminus\{0\}$.
\end{lem}

\begin{proof}
The linear form $\langle c,q\rangle$ has a minimum, $m_c$, on the compact set $\{q \in\bar{\mathfrak{P}}_+ \mid \|q\|_\infty = 1\}$. Since $c \in \mathfrak{C}_+$, $m_c>0$.  Hence for an arbitrary $q \in\bar{\mathfrak{P}}_+\setminus\{0\}$ we have
\begin{displaymath}
\langle c, q \rangle =  \langle c, \frac{q}{\|q\|_\infty} \rangle \|q\|_\infty \geq m_c\|q\|_\infty.
\end{displaymath}
Next we observe that 
\begin{displaymath}
\|Q\|_\infty = \max_{x \in K} Q = \max_{x \in K} \left\{ \sum_{k = 0}^{n} q_k \alpha_k\right\}
\leq \sum_{k = 0}^{n} |q_k| \, |\alpha_k| \leq  M \|q\|_\infty ,
\end{displaymath}
where $M:=\max_{k,x} |\alpha_k|$. This maximum exists and is positive and finite, since the basis functions $\alpha_1,\alpha_2,\dots,\alpha_n$ are continuous and $K$ is compact. Consequently, taking $\epsilon_c := m_c/M>0$, we obtain
\begin{displaymath}
\langle c, q \rangle \geq \epsilon_c \|Q\|_\infty.
\end{displaymath}
Next we consider the integral part of $\mathbb{J}^c_p$, namely
\begin{displaymath}
\int_{K} P \log Q dx = \int_{K} P \log\left(\frac{Q}{\|Q\|_\infty}\right) dx +  \log\|Q\|_\infty\int_{K} Pdx.
\end{displaymath}
Since $p\in\bar{\mathfrak{P}}_+\setminus\{0\}$ is fixed, $\epsilon_p:=\int_{K} Pdx >0$. Therefore
\begin{displaymath}
\mathbb{J}_p^c(q) \geq \epsilon_c \|Q\|_\infty - \int_{K} P \log\left(\frac{Q}{\|Q\|_\infty}\right) dx - \epsilon_p \log\left( \|Q\|_\infty \right).
\end{displaymath}
However, the integrand in the second term is nonpositive, and hence \eqref{Jestimate} follows. 
\end{proof}

\begin{lem}\label{lem:Jcontinuity}
Let $c\in\mathfrak{C}_+$, and let $\mathbb{J}^c_p(q):\,\bar{\mathfrak{P}}_+\to\mathbb{R}\cup\{\infty\}$ be the functional \eqref{eq:objFun}.  Then the optimal value $\mathbb{J}^c_p(\hat q)=\min_{q \in \bar{ \mathfrak{P} }_+} \mathbb{J}^c_p(q)$ is continuous in $p$ over $\bar{\mathfrak{P}}_+\setminus\{0\}$.
\end{lem}

\begin{proof}
Let $p_1, p_2\in\bar{\mathfrak{P}}_+\setminus\{0\}$ be arbitrary, and let $q_1, q_2\in\bar{\mathfrak{P}}_+$ be the unique minimizers of $\mathbb{J}^c_{p_1}$ and $\mathbb{J}^c_{p_2}$, respectively. Choose a $q_0 \in \mathfrak{P}_+$ such that $0 < \|Q_0\|_\infty < \infty$. Then, by optimality,
\begin{subequations}\label{optimalitycond}
\begin{eqnarray}
&\mathbb{J}^c_{p_1}(q_1)\leq \mathbb{J}^c_{p_1}(q_2+\varepsilon q_0)& \\
&\mathbb{J}^c_{p_2}(q_2)\leq \mathbb{J}^c_{p_2}(q_1+\varepsilon q_0)&
\end{eqnarray}
\end{subequations}
for all $\varepsilon >0$.  Therefore, if we could show that, for any $\delta >0$,
\begin{subequations}\label{variation}
\begin{eqnarray}
&|\mathbb{J}^c_{p_2}(q_1+\varepsilon q_0) - \mathbb{J}^c_{p_1}(q_1)|\leq \delta & \label{variation1}\\
&|\mathbb{J}^c_{p_1}(q_2+\varepsilon q_0)-\mathbb{J}^c_{p_2}(q_2)|\leq \delta &\label{variation2}
\end{eqnarray}
\end{subequations}
for $\|p_2 -p_1\|$ sufficiently small, we would have 
\begin{displaymath}
\mathbb{J}^c_{p_2}(q_2)-\delta\leq \mathbb{J}^c_{p_1}(q_1)\leq \mathbb{J}^c_{p_2}(q_2)+\delta,
\end{displaymath}
and the lemma would follow. To prove this, form
\begin{displaymath}
\begin{split}
&|\mathbb{J}^c_{p_2}(q_1+\varepsilon q_0) - \mathbb{J}^c_{p_1}(q_1)|\\
&\phantom{xx}=\left|\langle c, \varepsilon q_0 \rangle - \int_{K} P_2 \log(Q_1 + \varepsilon Q_0) dx+\int_{K} P_1 \log Q_1 dx \right|\\
&\phantom{xx}=\left|\langle c, \varepsilon q_0 \rangle - \int_{K} P_1 \log\left(1 + \frac{\varepsilon Q_0}{Q_1}\right) dx  - \int_{K} (P_2 - P_1) \log(Q_1 + \varepsilon Q_0) dx\right| \\
&\phantom{xx}\leq \varepsilon \left ( \langle c, q_0 \rangle + \int_{K} P_1 \frac{Q_0}{Q_1} dx \right)
+ \|P_2 - P_1\|_1 \|\log(Q_1 + \varepsilon Q_0)\|_\infty,
\end{split}
\end{displaymath}
where, by optimality of $q_1$ with respect to $\mathbb{J}^c_{p_1}$,  $\int_{K} P_1 \frac{Q_0}{Q_1} dx=\langle c,q_0\rangle$ is finite. 
Here we can make the first term less or equal to $\delta/2$ by choosing $\varepsilon$ sufficiently small. Then, set $K(\varepsilon):=2\|\log(Q_1 + \varepsilon Q_0)\|_\infty$, and take $\|P_2 - P_1\|_1\leq \delta/K(\varepsilon)$, from which \eqref{variation1} follows. The inequality \eqref{variation2} follows by the same line of argument. 
\end{proof}

\begin{prop}\label{prop:p2qhatcontinuity}
Let $\mathbb{J}^c_p(q):\,\bar{\mathfrak{P}}_+\to\mathbb{R}\cup\{\infty\}$ be the functional \eqref{eq:objFun}. Then the function
\begin{equation*}
\hat{q} = \argmin_{q \in \bar{ \mathfrak{P} }_+} \mathbb{J}^c_p(q)
\end{equation*}
that maps $p \in \bar{ \mathfrak{P} }_+ \setminus \{ 0 \}$ to $\hat{q} \in \bar{ \mathfrak{P} }_+$ is  continuous for all values of $c \in \mathfrak{C}_+$. 
\end{prop}

\begin{proof}
Let $(p_k)$ be a sequence in $\mathfrak{P}_+$ converging to $p \in \bar{\mathfrak{P}}_+ \setminus \{0\}$ as $k\to\infty$. Moreover, let $q_k = \argmin_{q \in \bar{\mathfrak{P}}_+} \mathbb{J}^c_{p_k}(q)$ and $\hat{q} = \argmin_{q \in \bar{\mathfrak{P}}_+} \mathbb{J}^c_{p}(q)$. Then, by Lemma~\ref{sublevelsetlem}, the sequence $(q_k)$ is bounded, and hence there is a subsequence, which we also call $(q_k)$ converging to a limit  $q_\infty$.  Assume that $ q_\infty \neq \hat{q}$, and choose a $q_0 \in \mathfrak{P}_+$ such that $0 < ||Q_0||_\infty < \infty$. Then
\begin{displaymath}
\begin{split}
\mathbb{J}^c_{p_k}(q_k) &= \mathbb{J}^c_{p_k} (q_k + \varepsilon q_0) - \langle c, \varepsilon q_0 \rangle + \int_{K} P_k \log \left( \frac{Q_k + \varepsilon Q_0}{Q_k} \right)dx\\
&\geq \mathbb{J}^c_{p_k} (q_k + \varepsilon q_0) - \langle c, \varepsilon q_0 \rangle .
\end{split}
\end{displaymath}
Therefore, by Lemma~\ref{lem:Jcontinuity}, we have 
\begin{equation*}
\mathbb{J}^c_{p}(\hat{q}) = \lim_{k \rightarrow \infty} \mathbb{J}^c_{p_k}(q_k) \geq \lim_{k \rightarrow \infty} \mathbb{J}^c_{p_k}(q_k + \varepsilon q_0) - \varepsilon \langle c, q_0 \rangle.
\end{equation*}
However $q_k + \varepsilon q_0 \in \mathfrak{P}_+$, and, since $(p,q)\mapsto\mathbb{J}^c_p(p,q)$ is continuous in $\mathfrak{P}_+$, we obtain
\begin{equation}\label{Jpestimate}
\mathbb{J}^c_{p}(\hat{q}) \geq \lim_{k \to \infty} \mathbb{J}^c_{p_k}(q_k + \varepsilon q_0) - \varepsilon \langle c, q_0 \rangle = \mathbb{J}^c_{p}(q_\infty + \varepsilon q_0) - \varepsilon \langle c, q_0 \rangle
\end{equation}
Since the minimum $\min_{q \in \bar{ \mathfrak{P} }_+} \mathbb{J}^c_{p}(q)$ is unique (Theorem~\ref{uniqueminthm}), there is a $\delta > 0$ such that  $\mathbb{J}^c_{p}(q_\infty) - \delta > \mathbb{J}^c_{p}(\hat{q})$, which together with \eqref{Jpestimate} yields 
\begin{displaymath}
\mathbb{J}^c_{p}(q_\infty) - \delta >  \mathbb{J}^c_{p}(q_\infty + \varepsilon q_0) - \varepsilon \langle c, q_0 \rangle
\end{displaymath}
However by letting $\varepsilon \to 0$ we get $- \delta > 0$, which is a contradiction. Therefore $\lim_{k \to \infty} q_k = \hat{q}$, as claimed.
\end{proof}

\begin{prop}\label{prop:c2qhatcontinuity}
Let $\hat{q} := \argmin_{q \in \bar{ \mathfrak{P} }_+} \mathbb{J}^c_p(q)$, where $\mathbb{J}^c_p(q):\,\bar{\mathfrak{P}}_+\to\mathbb{R}\cup\{\infty\}$ is the functional \eqref{eq:objFun}. Then the function sending $c\in\mathfrak{C}_+$ to $\hat{q}$ is continuous for all values of $p \in \bar{ \mathfrak{P} }_+ \setminus \{ 0 \}$.
\end{prop}

\begin{proof}
The proof is analogous to that of Proposition~\ref{prop:p2qhatcontinuity}. We first prove that the optimal value $\mathbb{J}^c_p(\hat q)=\min_{q \in \bar{ \mathfrak{P} }_+} \mathbb{J}^c_p(q)$ is continuous in $c$ over $\mathfrak{C}_+$. To this end, let $c_1, c_2\in\mathfrak{C}_+$ be arbitrary, and let $q_1, q_2\in\bar{\mathfrak{P}}_+$ be the unique minimizers of $\mathbb{J}^{c_1}_p$ and $\mathbb{J}^{c_2}_p$, respectively. Next, exchanging $\mathbb{J}_{p_1}^c$ for $\mathbb{J}_p^{c_1}$ and $\mathbb{J}_{p_2}^c$ for $\mathbb{J}_p^{c_2}$ everywhere in the proof of Lemma~\ref{lem:Jcontinuity}, we see that it just remains to show that the absolute value of 
\begin{displaymath}
\mathbb{J}^{c_2}_p(q_1+\varepsilon q_0) - \mathbb{J}^{c_1}_p(q_1)\leq\|q_1\|\|c_2 -c_1\| +\varepsilon\left[\langle c_1,q_0\rangle +\int_K P\frac{Q_0}{Q_1}dx\right]
\end{displaymath}
can be made smaller than some preselected $\delta >0$ by choosing an appropriately small $\varepsilon$. This follows from the same argument as in Lemma~\ref{lem:Jcontinuity}, thus establishing the continuity of the optimal value.

Next, let $(c_k)$ be a sequence in $\mathfrak{C}_+$ converging to $c\in\mathfrak{C}_+$ as $k\to\infty$, and let 
 $q_k = \argmin_{q \in \bar{\mathfrak{P}}_+} \mathbb{J}^{c_k}_p(q)$ and $\hat{q} = \argmin_{q \in \bar{\mathfrak{P}}_+} \mathbb{J}^c_{p}(q)$. Since $(q_k)$ is bounded, there is a subsequence, also called $(q_k)$ such that $q_k\to q_\infty$. We assume that $q_\infty\neq \hat{q}$, and show that this leads to a contradiction. As in the proof of Proposition~\ref{prop:p2qhatcontinuity}, we first note that 
 \begin{displaymath}
\mathbb{J}^{c_k}_p(q_k)\geq\mathbb{J}^{c_k}_p(q_k+\varepsilon q_0) -\varepsilon \langle c_k,q_0\rangle
\end{displaymath}
and hence, by the continuity of the optimal value of $\mathbb{J}^c_{p}$ in $c$, that
\begin{displaymath}
\begin{split}
\mathbb{J}^c_{p}(\hat{q}) = \lim_{k \to \infty} \mathbb{J}^{c_k}_p(q_k) &\geq \lim_{k \to \infty} \left(\mathbb{J}^{c_k}_p(q_k + \varepsilon q_0) - \varepsilon \langle c_k, q_0 \rangle\right)\\
&=\mathbb{J}^{c}_p(q + \varepsilon q_0) - \varepsilon \langle c, q_0 \rangle ,
\end{split}
\end{displaymath}
where we have used the fact that  $\mathbb{J}^c_{p}$ is continuous in  $(c,q)\in\mathfrak{C}_+\times\mathfrak{P}_+$. Then  letting $\varepsilon \to 0$ leads to a contradiction, proving that $q_\infty=\hat{q}$, as claimed.
\end{proof}

\section{Errata}\label{errata}
The proof of the Lemma 3.1 in \cite{byrnes2006thegeneralized} contains an error emanating from an incorrect use of the dominate convergence theorem, which invalidates Theorem 1.10 in \cite{byrnes2006thegeneralized}. The correct statement is the one-dimensional version of Theorem~\ref{thm:mainMoment} in the present paper.

\bibliographystyle{plain}
\bibliography{ref}

\end{document}